\documentclass[10pt,fleqn]{article}
\usepackage{amsmath,amsthm,amssymb,amscd}
\usepackage[centering,
 textwidth=1.2\textwidth,
 textheight=1.1\textheight]{geometry}
\newcommand{\texorpdfstring}[2]{#1} 

\makeatletter
     
     \DeclareMathAlphabet{\mathbf}{OT1}{cmr}{b}{n}
     \def\fo@ter{\hss\footnotesize\bfseries\thepage\hss}
     \def\ps@myheadings{%
      \let\@oddfoot\fo@ter
      \let\@evenfoot\@oddfoot
      \def\@oddhead{\footnotesize\scshape\hfil\rightmark\hfil}%
      \let\@evenhead\@oddhead
      \let\@mkboth\@gobbletwo
      \let\sectionmark\@gobble
      \let\subsectionmark\@gobble}
     \def\ps@plain{\let\@mkboth\@gobbletwo
      \let\@oddhead\@empty
      \let\@oddfoot\fo@ter
      \let\@evenhead\@empty
      \let\@evenfoot\@oddfoot}
     \renewcommand\section{\@startsection {section}{1}{\z@}%
       {-3.5ex \@plus -1ex \@minus -.2ex}%
       {2.3ex \@plus.2ex}%
       {\normalfont\large\bfseries}}
     \renewcommand\subsection{\@startsection{subsection}{2}{\z@}%
       {-3.25ex\@plus -1ex \@minus -.2ex}%
       {1.5ex \@plus .2ex}%
       {\normalfont\normalsize\bfseries}}
     \renewcommand\subsubsection{\@startsection{subsubsection}{3}{\z@}%
       {-3.25ex\@plus -1ex \@minus -.2ex}%
       {1.5ex \@plus .2ex}%
       {\normalfont\normalsize\scshape}}
     \renewcommand\paragraph{\@startsection{paragraph}{4}{\z@}%
       {3.25ex \@plus1ex \@minus.2ex}%
       {-1em}%
       {\normalfont\normalsize\scshape}}
     \def\@maketitle{%
      \newpage
      \null
      \rightline{\footnotesize\@date}%
      \vskip 2em%
     \begin{center}%
      \let \footnote \thanks
       {\Large\@title\par}%
       \vskip 1.5em%
       {\large
        \lineskip .5em%
        \begin{tabular}[t]{c}%
         \normalfont\@author
        \end{tabular}\par}%
     \end{center}%
      \par
      \vskip 1.5em}
     \if@titlepage
      \renewenvironment{abstract}{%
       \titlepage
       \null\vfil
       \@beginparpenalty\@lowpenalty
       \begin{center}%
        \bfseries \abstractname
        \@endparpenalty\@M
       \end{center}}%
      {\par\vfil\null\endtitlepage}
     \else
      \renewenvironment{abstract}{%
       \if@twocolumn
        \section*{\abstractname}%
       \else
        \small
        \quotation
        \noindent{\bfseries\abstractname.}\quad\ignorespaces
       \fi}
      {\if@twocolumn\else\endquotation\fi}
     \fi

     \renewenvironment{proof}[1][\proofname]{\par
        \pushQED{\qed}%
        \normalfont \topsep6\p@\@plus6\p@\relax
        \trivlist
        \item[\hskip\labelsep
              \scshape 
          \hskip\parindent#1\@addpunct{.}]\ignorespaces 
      }{%
        \popQED\endtrivlist\@endpefalse
     }
\makeatother

\newtheoremstyle{proclamation}
  {\topsep}
  {\topsep}
  {\slshape}
  {\parindent}
  {\scshape}
  {.}
  {.5em}
  {}
\theoremstyle{proclamation}
\newtheorem{lemma}{Lemma}
\newtheorem{proposition}{Proposition}

\newtheorem*{main-theorem}{Theorem}
\newtheorem*{corollary}{Corollary}

\newtheoremstyle{remark}
  {\topsep}
  {\topsep}
  {\normalfont}
  {\parindent}
  {\scshape}
  {.}
  {.5em}
  {}
\theoremstyle{remark}
\newtheorem{definition}{Definition}
\newtheorem{remark}{Remark}

\newcommand{\onehalf}{{\frac12}}
\newcommand{\C}{\mathbf{C}}
\newcommand{\Ctimes}{\C^\times}
\newcommand{\Q}{\mathbf{Q}}
\newcommand{\R}{\mathbf{R}}
\newcommand{\Z}{\mathbf{Z}}
\renewcommand{\O}{\mathcal{O}} 
\newcommand{\torus}{\mathbf{T}}
\newcommand{\hilb}{\mathcal{H}}
\newcommand{\bo}{\mathcal{L}} 
\newcommand{\tc}{{\mathcal{L}^1}} 
\newcommand{\hs}{{\mathcal{L}^2}} 
\newcommand{\detop}{\mathcal{T}} 
\newcommand{\detopone}{\detop_1} 
\newcommand{\ck}{\tau} 
\newcommand{\bd}{\partial} 
\newcommand{\fm}[1]{\mathcal{M}^{#1}} 
\newcommand{\fmext}{\mathcal{E}} 
\newcommand{\mult}[1]{M({#1})} 
\newcommand{\toeplitz}[1]{T({#1})} 
\newcommand{\vna}{\mathcal{N}} 
\newcommand{\pairing}[2]{\langle#1,#2\rangle}
\newcommand{\norm}[1]{\|#1\|}
\newcommand{\hsnorm}[1]{\norm{#1}_{\hs}} 
\newcommand{\symb}[2]{\{#1,#2\}} 
\newcommand{\addcommutator}[2]{{#1}{#2}-{#2}{#1}} 
\newcommand{\multcommutator}[2]{[#1,#2]} 
\newcommand{\jt}[2]{\jointtorsion(#1,#2)} 
\newcommand{\ft}[1]{\widehat{#1}} 
\newcommand{\ftoperator}{\mathcal{F}} 
\newcommand{\crossprod}{\ltimes} 
\newcommand{\variable}{{}\cdot{}} 
\DeclareMathOperator{\pv}{P.V.} 
\DeclareMathOperator{\ord}{ord}
\DeclareMathOperator{\Hom}{Hom}
\DeclareMathOperator{\Res}{Res} 
\DeclareMathOperator{\restrictionmap}{res} 
\newcommand{\res}[2]{\restrictionmap_{#1}^{#2}}
\DeclareMathOperator{\ad}{Ad} 
\DeclareMathOperator{\tr}{Tr}
\DeclareMathOperator{\jointtorsion}{\tau_{\textrm{CP}}}
\DeclareMathOperator{\Spec}{Spec}
\DeclareMathOperator{\Diff}{Diff}
\DeclareMathOperator{\GL}{GL}
\DeclareMathOperator{\steinberg}{St} 
\renewcommand{\phi}{\varphi}
\let\bs\backslash
\let\cnx\nabla
\let\isom\cong
\let\comp\circ
\let\injlim\varinjlim
\let\longto\longrightarrow
\let\injection\hookrightarrow

\begin{document}

\title{Fredholm modules and the Beilinson--Bloch regulator}
\author{Eugene Ha and Victor Kasatkin}
\date{\empty}
\pagestyle{myheadings}
\markright{Fredholm modules and the Beilinson--Bloch regulator\hfill
  E.~Ha, V.~Kasatkin}

\maketitle

\begin{abstract}
We prove an operator-theoretic reconstruction of the Beilinson--Bloch regulator for compact Riemann surfaces, using loop operators and the Connes--Karoubi character for Fredholm modules. The proof includes a new computation of the Connes--Karoubi character for Steinberg symbols of the circle, which relies on the Helton--Howe determinant theory, but not on the Carey--Pincus theory of joint torsion.
\end{abstract}

\section{Introduction}

The main objective of this article is to prove a reconstruction of the Beilinson--Bloch regulator for compact Riemann surfaces in terms of the Connes--Karoubi character for Fredholm modules. In this introduction, we outline our construction and clarify our contribution in relation to other works. Our work owes much to the work of Connes and Karoubi \cite{Connes-Karoubi:1984}, Pressley, Segal, and Wilson \cite{Segal-Wilson:1985}, \cite{Pressley-Segal:1986}, and Helton and Howe \cite{Helton-Howe:1973}.

\subsection{The Beilinson--Bloch regulator via the Connes--Karoubi character (an overview)}

Let $X$ be a compact Riemann surface. The Beilinson--Bloch regulator of~$X$ is a Chern character map
\begin{equation}
\label{eq:full-regulator}
r\colon K_2(X)\longto H^1(X,\Ctimes).
\end{equation}
It was first constructed for elliptic curves by Bloch \cite{Bloch:2000}, and for arbitrary compact Riemann surfaces by Beilinson \cite{Beilinson:1980}, as just one instance of his far-reaching conjectures relating algebraic $K$-theory of varieties (defined over~$\Q$) to special values of $L$-functions \cite{Beilinson:1984}. The essence of the construction of the regulator~\eqref{eq:full-regulator}, which will be reviewed in~\S\ref{sec:review-bb-regulator}, is its construction on the Zariski open subsets of~$X$, i.e., the construction of a family of homomorphisms
\[
r_S\colon K_2(\O(X\setminus S))
\longto\Hom(\pi_1(X\setminus S),\Ctimes)
\isom H^1(X\setminus S,\Ctimes),
\]
where $S$ ranges over the finite subsets of~$X$, and $\O(X\setminus S)$ is the ring of analytic functions on~$X\setminus S$. (We will suppress base points in this introduction.)

For a smooth loop $\gamma\colon S^1\to X\setminus S$, one gets a representation of the ring $\O(X\setminus S)$ on the Hilbert space~$L^2(S^1)$ via multiplication operators:
\[
\mu_\gamma\colon\O(X\setminus S)\longto \bo(L^2(S^1)),
\quad
f\mapsto\mult{f\comp\gamma}.
\]
The key observation here---due to Pressley, Segal, and Wilson \cite{Segal-Wilson:1985}, \cite{Pressley-Segal:1986}---is that the representation $\mu_\gamma$ is a $2$-summable Fredholm module \cite{Connes-Karoubi:1984}, \cite{Connes-Karoubi:1988}. The Connes--Karoubi character for Fredholm modules, ibid., then furnishes a homomorphism
\[
\ck_{\mu_\gamma}\colon K_2(\O(X\setminus S))\longto\Ctimes.
\]
Our main theorem is the following result (see \S\ref{sec:main-theorem}).

\begin{main-theorem}
For every $u\in K_2(\O(X\setminus S))$ and every $[\gamma]\in\pi_1(X\setminus S)$, we have
\begin{equation}
\label{eq:reconstruction-formula}
\pairing{r_S(u)}{[\gamma]}=\ck_{\mu_\gamma}(u),
\end{equation}
where $\gamma\colon S^1\to X\setminus S$ is any smooth loop in the homotopy class~$[\gamma]$.
\end{main-theorem}

In particular, one deduces the non-obvious fact that the character $\ck_{\mu_\gamma}$ depends only on the homotopy class of~$\gamma$, and is multiplicative in~$[\gamma]$, because such is the case for the Beilinson--Bloch regulator.

The theorem will be proved in two steps. First, the equality~\eqref{eq:reconstruction-formula} is established when $u$ is a Steinberg symbol (Proposition~\ref{prop:bb-regulator-agreement-on-steinberg-symbols}); our proof of this known fact is new, and, unlike other proofs in the literature, does not depend on establishing equality with the Carey--Pincus joint torsion. Second, the continuity of the map $\gamma\mapsto\ck_{\mu_\gamma}(u)$ (Lemma~\ref{lemma:continuity-regulator-loop-variable}) is exploited to establish equality~\eqref{eq:reconstruction-formula} for \emph{arbitrary} elements $u\in K_2(\O(X\setminus S))$ (\S\ref{sec:main-theorem}). In this way, a direct (and possibly tricky) proof of the homotopy invariance of $\ck_{\mu_\gamma}$, and of the multiplicativity in~$[\gamma]$, is overcome.

\subsection{Related works}
\label{sec:related-works}

Since the seminal work of Karoubi \cite{Karoubi:1982}, \cite{Karoubi:1983}, \cite{Karoubi:1987}, and Connes and Karoubi \cite{Connes-Karoubi:1984}, \cite{Connes-Karoubi:1988}, on relative $K$-theory and its Chern character map to cyclic homology, it has been clear that the Connes--Karoubi character is relevant to the study of regulators. Indeed, Karoubi himself compared Borel's regulator to the Connes--Karoubi character for~$\C$ \cite{Karoubi:1987}. Recently, Tamme has constructed an analogue of Karoubi's relative Chern character, and has established a precise comparison to Beilinson's Chern character \cite{Tamme:2012}.

Prior to the present work, connections between the Beilinson--Bloch regulator and the Connes--Karoubi character had been identified in two settings:

\begin{enumerate}
\item Bloch \cite[\S2]{Bloch:1981} generalized the canonical (infinite-dimensional) Lie algebra extension corresponding to the Beilinson--Bloch regulator (cf.~first remark in~\cite{Beilinson:1980}). The basic ingredient of Bloch's extension is an arbitrary ring~$R$ containing $\onehalf$. When $R=C^\infty(S^1)$, Delabre observed \cite{Delabre:1998} that Bloch's extension coincides with the Lie algebra extension corresponding to the Connes--Karoubi character $\ck_{S^1}\colon K_2(C^\infty(S^1))\to\Ctimes$ \cite[5.9]{Connes-Karoubi:1988}.
\item Carey and Pincus identified a connection between the Beilinson--Bloch regulator and Toeplitz operator determinants \cite{Carey-Pincus:1999}. Specifically, they showed \cite[\S5]{Carey-Pincus:1999} that the joint torsion of Toeplitz operators is given by Beilinson's monodromy formula~\eqref{eq:monodromy-formula}.
\end{enumerate}

A key computation in our work is the evaluation of the Connes--Karoubi character $\ck_{S^1}$ on Steinberg symbols $\symb fg\in K_2(C^\infty(S^1))$. This is known to equal the joint torsion of Toeplitz operators \cite[Prop.~1]{Carey-Pincus:1999}, thanks to a combination of recent results of Kaad \cite{Kaad:2011a} and Migler \cite{Migler:2014} that enable the joint torsion, in this case, to be matched with the Connes--Karoubi character. (For details, see the remark at the end of~\S\ref{sec:R_S-steinberg-symbols}.) However, since our work does not concern joint torsion, we choose to execute a direct, and considerably shorter, computation of $\ck_{S^1}\symb fg$, which employs the determinant theory of Helton and Howe \cite{Helton-Howe:1973}, but is otherwise fairly elementary. The validity of the reconstruction formula~\eqref{eq:reconstruction-formula} for the full $K$-group is, to the best of our knowledge, a new result.

Lastly, let us mention that analogues of $\ck_{S^1}\symb fg$ have been computed for higher odd-dimensional manifolds (with Dirac operator) by Kaad \cite{Kaad:2011b} and, in much greater generality, Bunke \cite{Bunke:2014}.

\subsection{Organization of this article}

The Beilinson--Bloch regulator and the Connes--Karoubi character are reviewed in \S\S\ref{sec:review-bb-regulator}, \ref{sec:review-ck-character}; formula~\eqref{eq:monodromy-formula} and Proposition~\ref{prop:ck-formula} summarize the main facts that are required later. A proof of the Pressley--Segal--Wilson theorem on the fundamental Fredholm structure of loop operators is recalled in \S\ref{sec:loop-operators}. The reconstruction of the regulator from the Connes--Karoubi character is carried out in~\S\ref{sec:reconstruction-bb-regulator}. Some open questions, and our motivation for undertaking this work, are discussed in~\S\ref{sec:conclusion}.

\paragraph*{Notational conventions.}

We denote by $X$ a fixed compact Riemann surface, and by $S$, $S'$, etc., finite sets of points of~$X$. The standard notation for commutators is reserved for the multiplicative case: $\multcommutator ab=aba^{-1}b^{-1}$. (Additive commutators, $\addcommutator ab$, will be written out in full.) The Toeplitz operator with symbol~$f$ is denoted by~$\toeplitz f$.

\subsection*{Acknowledgments}

This work is derived from a 2012 preprint by E.H., which, however, was marred by technical errors and omissions. These were later fixed in the Ph.D.\ thesis of V.K.\ \cite{Kasatkin:2015}, which also provided the proof of our main theorem.

E.H.\ is grateful to Andreas Thom for pointing out a gap in the original 2012 preprint, for bringing the work of Carey and Pincus to his attention, for helpful discussions on Deninger's work (which, in particular, motivated the present investigation; see~\S\ref{sec:entropy-regulators-determinants}), and, above all, for generously funding a position that enabled the initial work. V.K.\ thanks his advisor, Prof.\ Matilde Marcolli, for her guidance and expert advice. Both authors are indebted to her for catalyzing their collaboration.

\section{The Beilinson--Bloch regulator}
\label{sec:review-bb-regulator}

We review Beilinson's construction of the regulator for compact Riemann surfaces. There are two parts to this: a general framework (\S\ref{sec:abstract-framework-regulator}), and an explicit realization of it (formula~\eqref{eq:monodromy-formula}).

\subsection{General framework for the regulator}
\label{sec:abstract-framework-regulator}

The main task in constructing the regulator~$r$ \eqref{eq:full-regulator} is its construction on the generic point~$\eta=\Spec\C(X)$, i.e., the construction of a homomorphism
\begin{equation}
\label{eq:regulator-generic-point}
r_\eta\colon K_2(\C(X))\longto
H^1(\Spec\C(X),\Ctimes)
:=\injlim\nolimits_SH^1(X\setminus S,\Ctimes),
\end{equation}
where $\C(X)$ is the field of meromorphic functions on~$X$, and $S$ runs over (increasing) finite subsets of~$X$. The function field~$\C(X)$ is canonically isomorphic to the stalk of the structure sheaf at the generic point, i.e.,
\begin{equation}
\label{eq:generic-point-stalk}
\C(X)\isom\injlim\nolimits_S\O(X\setminus S),
\end{equation}
where $\O(X\setminus S)$ is the ring of analytic functions on~$X\setminus S$. Since $K_2$ is a continuous functor, we have
\[
K_2(\C(X))\isom\injlim\nolimits_SK_2(\O(X\setminus S)).
\]
Therefore constructing $r_\eta$ boils down to constructing an $S$-compatible family of homomorphisms
\begin{equation}
\label{eq:S-regulator}
r_S\colon K_2(\O(X\setminus S))\longto H^1(X\setminus S,\Ctimes).
\end{equation}
We may then set
\[
r_\eta=\injlim\nolimits_S r_S
\colon K_2(\C(X))\longto H^1(\Spec\C(X),\Ctimes).
\]

Granting the construction of~$r_\eta$, the regulator on~$X$ \eqref{eq:full-regulator} comes from the following diagram with exact rows:
\[
\begin{CD}
@. K_2(X) @>>> K_2(\C(X)) @>{\oplus\tau_x}>> \bigoplus\limits_{x\in X}\Ctimes \\
@. @. @VV{r_\eta}V @| \\
0 @>>> H^1(X,\Ctimes) @>>> H^1(\Spec\C(X),\Ctimes)
  @>{\oplus\Res_x}>> \bigoplus\limits_{x\in X}\Ctimes
\end{CD}
\]
The top row is the localization sequence in $K$-theory induced by the inclusion $\Spec\C(X)\to X$. The quantity $\tau_x\symb fg=(-1)^{\ord_xf\ord_xg}(f^{\ord_xg}/g^{\ord_xf})(x)$ is the tame symbol. The bottom row is the Gysin sequence in cohomology (here $\Res_x$ is the residue map at~$x$). The regulator~$r_\eta$ on the generic point connects the two rows, making the square commute \cite[Prop.~1.19]{Bloch:1981}. Therefore, because the rows are exact, there is an induced homomorphism
\[
r\colon K_2(X)\longto H^1(X,\Ctimes)
\]
maintaining the commutativity of the full diagram. This is the \emph{Bei\-lin\-son--Bloch regulator} of~$X$.

\subsection{Beilinson's construction of the regulator}

Beilinson~\cite{Beilinson:1980} gave a direct construction of~$r_\eta$---only implicitly defining the family~$\{r_S\}_S$ \eqref{eq:S-regulator}---by making critical use of Matsumoto's Theorem on $K_2$ of fields (see \cite[4.3.15]{Rosenberg:1994}), and of the isomorphism
\begin{equation}
\label{eq:monodromy-mapping}
H^1(X\setminus S,\Ctimes)\isom\Hom(\pi_1(X\setminus S,x_0),\Ctimes)
\end{equation}
given by the monodromy mapping, upon interpreting the cohomology as the group of flat line bundles on~$X\setminus S$. (The isomorphism is canonical up to the choice of base point~$x_0$.)

Matsumoto's Theorem asserts that $K_2(\C(X))$ is generated by the set of all Steinberg symbols $\symb fg$, with $f$, $g\in\C(X)^\times$. For a given Steinberg symbol~$\symb fg$, take $S\subset X$ to be large enough to contain all the poles and zeros of~$f$ and~$g$. Beilinson then defined a character $r_\eta(f,g)\in\Hom(\pi_1(X\setminus S,x_0),\Ctimes)$ by the pairing
\begin{equation}
\label{eq:monodromy-formula}
\pairing{r_\eta(f,g)}{[\gamma]}
=\exp\left\{\frac1{2\pi i}
\left(\int_\gamma\log f\,d\log g
-\log g(x_0)\int_\gamma d\log f\right)\right\}
\end{equation}
for $\gamma$ a smooth loop in~$(X\setminus S,x_0)$---neither the choice of branches of~$\log f$ and $\log g$, nor the choice of smooth representative of~$[\gamma]$, affect the quantity~\eqref{eq:monodromy-formula}. The character $r_\eta(f,g)$ is bi-multiplicative, skew-symmetric, and satisfies the Steinberg relation $\pairing{r_\eta(f,1-f)}{[\gamma]}=1$ for all loops in~$(X\setminus S,x_0)$ that also avoid the zeros of~$1-f$. Therefore $r_\eta(f,g)$ is determined by the Steinberg symbol $\symb fg$, and in this way the regulator~\eqref{eq:regulator-generic-point} is determined on the generic point of~$X$. Note that it is compatible with the limit~\eqref{eq:generic-point-stalk}, and so a family of homomorphisms~\eqref{eq:S-regulator} is thereby implicitly obtained: when $S$ contains all the poles and zeros of~$f$ and~$g$, the Steinberg symbol $\symb fg$ may be regarded as an element of $K_2(\O(X\setminus S))$, and $r_S\symb fg$ is again determined by formula~\eqref{eq:monodromy-formula}.

\begin{remark}
The articles of Bloch \cite{Bloch:1981} and Ramakrishnan \cite{Ramakrishnan:1981} prove these assertions by means of an alternative construction of $r_\eta$ in terms of the Heisenberg group, cf.~\S\ref{sec:heisenberg-group-construction}.
\end{remark}

\section{The Connes--Karoubi character for 2-summable Fredholm modules}
\label{sec:review-ck-character}

In this section, we review the definition of the Connes--Karoubi character in the special case of $2$-summable Fredholm modules \cite[5.3]{Connes-Karoubi:1988}. Much of the basic notation used subsequently is established in this section.

\subsection{2-Summable Fredholm modules}
\label{sec:2-summable-fredholm-modules}

Let $\hilb_+$  and $\hilb_-$ be (infinite-dimensional) separable Hilbert spaces. Let $P_+$, resp.\ $P_-$, denote the (orthogonal) projection of~$\hilb_+\oplus\hilb_-$ onto $\hilb_+$, resp.\ $\hilb_-$. Let $F=P_+-P_-$.

\begin{definition}[Connes--Karoubi {\cite[1.4]{Connes-Karoubi:1988}}]
A \emph{2-summable Fredholm module} is a complex algebra~$A$ together with a representation $\rho\colon A\to\bo(\hilb_+\oplus\hilb_-)$ such that the commutator $\addcommutator F{\rho(a)}$ is a Hilbert--Schmidt operator for every~$a\in A$. (The anti-diagonal components $P_\pm\rho(a)P_\mp$ are Hilbert--Schmidt operators.)

There is a \emph{universal 2-summable Fredholm module}
\[
\fm1=\{\,a\in\bo(\hilb_+\oplus\hilb_-)\colon\text{$\addcommutator Fa$ is a Hilbert--Schmidt operator}\,\}
\]
through which every 2-summable Fredholm module factors.
\end{definition}

The universal Fredholm module~$\fm1$ is a Banach algebra with norm $\norm{a}_{\fm1}=\norm{a}_\bo+\norm{\addcommutator Fa}_\hs$, where $\norm{\variable}_\bo$ is the operator norm and $\norm{\variable}_\hs$ is the Hilbert--Schmidt norm.

\subsection{The Connes--Karoubi character}
\label{sec:ck-character}

The Connes--Karoubi character for~$\fm1$ is a distinguished homomorphism
\begin{equation}
\label{eq:ck-character}
\ck\colon K_2(\fm1)\longto\Ctimes.
\end{equation}
It yields a character $K_2(A)\to\Ctimes$ for any 2-summable Fredholm module~$\rho\colon A\to\fm1$ via the composition
\[
\ck\comp K_2(\rho)\colon K_2(A)\longto\Ctimes.
\]

Before proceeding to the definition of~$\ck$, we recall some basic facts about $K_2$. Let $E$ be the group of elementary matrices, and let $\steinberg$ be the Steinberg group. Since the canonical projection $\steinberg(\fm1)\to E(\fm1)$ is a universal central extension whose kernel is, by definition, $K_2(\fm1)$, there is a canonical bijection between homomorphisms $K_2(\fm1)\to\Ctimes$ and (equivalence classes of) central extensions of~$E(\fm1)$ by~$\Ctimes$. Thus to define~$\ck\colon K_2(\fm1)\to\Ctimes$, it will suffice to define the corresponding central extension.

\subsubsection{The Connes--Karoubi central extension}

We now construct this central extension following~\cite[5.3--5]{Connes-Karoubi:1988}. To prevent notational clutter, we shall use the abbreviations $\bo$ for $\bo(\hilb_+)$, and $\tc$ for $\tc(\hilb_+)$, the ideal of trace class operators.

Let $\fmext$ be the fiber product $\fm1\times_{\bo/\tc}\bo$ of the homomorphism%
\footnote{This is a homomorphism---even though the map $a\mapsto P_+a P_+$ is not---because the product of Hilbert--Schmidt operators is of trace class.}
$\fm1\to\bo/\tc$, $a\mapsto P_+a P_+\bmod\tc$, and the projection $\bo\to\bo/\tc$:
\begin{equation}
\label{eq:basic-fiber-product}
\begin{CD}
\fmext @>{p_2}>> \bo \\
@VV{p_1}V @VVV \\
\fm1 @>>> \bo/\tc
\end{CD}
\end{equation}
The norm $\norm{(a,x)}_{\fmext}=\norm{a}_{\fm1}+\norm{P_+ a P_+-x}_{\tc}$ makes $\fmext$ a Banach algebra. We get an exact sequence of Banach algebras
\begin{equation}
\label{eq:main-exact-sequence}
0 \longto
\tc \longto
\fmext \xrightarrow{\,p_1\,}
\fm1 \longto
0
\quad
\text{(where $\tc\to\fmext\colon q\mapsto(0,q)$).}
\end{equation}

Let $\GL=\injlim_n\GL_n$. The key observation relating the exact sequence~\eqref{eq:main-exact-sequence} to $K_2$ is the following.

\begin{lemma}[Connes--Karoubi~{\cite[5.4]{Connes-Karoubi:1988}}]
The image of $\GL(p_1)$ is $E(\fm1)$.
\end{lemma}

Thus, applying the functor~$\GL$ to the sequence~\eqref{eq:main-exact-sequence}, we get the exact sequence
\begin{equation}
\label{eq:pre-ck-extension}
1\longto\detop
\longto\GL(\fmext)
\xrightarrow{\,\GL(p_1)\,}E(\fm1)
\longto1,
\end{equation}
where
\begin{equation}
\label{eq:operators-with-determinant}
\detop
=\ker\GL(p_1)
\isom\GL(p_2)(\detop)
=\ker\bigl(\GL(\bo)\to\GL(\bo/\tc)\bigr).
\end{equation}
The latter is the group of invertible operators with (Fredholm) determinant.\footnote{%
What we are calling here the ``Fredholm determinant'' is the extension to~$\ker\bigl(\GL(\bo)\to\GL(\bo/\tc)\bigr)$---thus to~$\detop$---of the usual Fredholm determinant for operators in  $\ker\bigl(\GL_1(\bo)\to\GL_1(\bo/\tc)\bigr)$.}

Let $\detopone=\ker(\det\colon\detop\to\Ctimes)$. This is a normal subgroup of~$\GL(\fmext)$, and the quotient $\GL(\fmext)/\detopone$ contains $\detop/\detopone$ in its center. The determinant furnishes an isomorphism $\detop/\detopone\isom\Ctimes$ of topological groups. Thus we get a central extension of~$E(\fm1)$ by~$\Ctimes$ upon dividing the extension~\eqref{eq:pre-ck-extension} by~$\detopone$.

\begin{definition}[Connes--Karoubi central extension {\cite[5.5]{Connes-Karoubi:1988}}]
The \emph{Connes--Karoubi character} is the homomorphism
\[
\ck\colon K_2(\fm1)\longto\Ctimes
\]
determined by the central extension
\begin{equation}
\label{eq:ck-central-extension}
1
\longto \Ctimes\isom\detop/\detopone
\longto \GL(\fmext)/\detopone
\xrightarrow{\,\GL(p_1)\,} E(\fm1)
\longto 1.
\end{equation}  
\end{definition}

\subsubsection{Formula for the Connes--Karoubi character}

More explicitly, the Connes--Karoubi character~$\ck$ is induced by the unique morphism~$\tilde\ck$ from the universal central extension of~$E(\fm1)$ to the central extension~\eqref{eq:ck-central-extension}:
\begin{equation}
\label{eq:ck-central-extension-homomorphism}
\begin{CD}
1 @>>> K_2(\fm1) @>>>\steinberg(\fm1) @>{\pi}>> E(\fm1) @>>> 1 \\
@. @VV{\tau}V @VV{\tilde\tau}V @| @. \\
1 @>>> \Ctimes @>>> \GL(\fmext)/\detopone @>{\GL(p_1)}>> E(\fm1) @>>> 1
\end{CD}
\end{equation}

With the aim of giving a formula for $\tilde\ck$, and therefore of~$\ck$, we recall some standard notation for the Steinberg group, cf.~\cite[\S4.2]{Rosenberg:1994}. Let $A$ be a ring. Let $x_{ij}(a)$ ($a\in A$, $i\neq j$), be the standard generators of~$\steinberg(A)$, and let $e_{ij}(a)$ be the corresponding elementary matrices. For $a$, $b\in\GL_1(A)$, let
\[
d_{12}(a)=
\begin{pmatrix}
  a & 0 & 0\\
  0 & a^{-1} & 0\\
  0 & 0 & 1
\end{pmatrix},
\quad
d_{13}(b)=
\begin{pmatrix}
  b & 0 & 0\\
  0 & 1 & 0\\
  0 & 0 & b^{-1}
\end{pmatrix}.
\]
Assume that~$a$ and~$b$ commute. Then $d_{12}(a)$ and $d_{13}(b)$ are commuting elements of~$E(A)$. Letting $\pi$ be the projection $\steinberg(A)\to E(A)$, the \emph{Steinberg symbol} of~$a$ and~$b$ is the (multiplicative) commutator
\begin{equation}
\label{eq:steinberg-symbol}
\symb ab=\multcommutator{\pi^{-1}\{d_{12}(a)\}}{\pi^{-1}\{d_{13}(b)\}}\in K_2(A).
\end{equation}
It is well-defined because the kernel of~$\pi$ is central.

\begin{proposition}[Formula for~$\ck$]
\label{prop:ck-formula}
In terms of the above notation, we have
\begin{align}
\label{eq:ck-character-formula-lift}
\tilde\ck(x_{ij}(a))
  &=e_{ij}((a,P_+a P_+))\detopone\in\GL(\fmext)/\detopone
  &&(a\in\fm1,\ i\neq j)\\
\label{eq:ck-character-formula}
\ck(u)
  &=\det\comp\GL(p_2)(\tilde\ck(u))
  &&(u\in K_2(\fm1))\\
\label{eq:ck-character-formula-steinberg-symbols}
\ck\symb ab
  &=\det\comp\GL(p_2)\bigl(
      \multcommutator{\GL(p_1)^{-1}\{d_{12}(a)\}}{\GL(p_1)^{-1}\{d_{13}(b)\}}\bigr)
  &&(a,\ b\in\GL_1(\fm1))
\end{align}
\end{proposition}

\begin{proof}
  Formula~\eqref{eq:ck-character-formula} follows from the combination of \eqref{eq:ck-central-extension}, \eqref{eq:ck-central-extension-homomorphism}, and~\eqref{eq:operators-with-determinant}.

To get formula~\eqref{eq:ck-character-formula-lift}, first note that for $x\in \steinberg(\fm1)$, every element of the set $\GL(p_1)^{-1}\{\pi(x)\}$ equals $\tilde\ck(x)$ up to multiplication by a central element of $\GL(\fmext)/\detopone$. Thus for $x$, $y\in\steinberg(\fm1)$, we have
\begin{equation}
\label{eq:ck-character-formula-multcommutator}
\tilde\ck(\multcommutator xy)
=\multcommutator{\GL(p_1)^{-1}\{\pi(x)\}}{\GL(p_1)^{-1}\{\pi(y)\}}.
\end{equation}
The generator~$x_{ij}(a)\in\steinberg(\fm1)$ satisfies the relation $x_{ij}(a)=\multcommutator{x_{ik}(1)}{x_{kj}(a)}$, for any $k\neq i$,~$j$ (see \cite[4.2.1(c)]{Rosenberg:1994}). Formula~\eqref{eq:ck-character-formula-lift} is therefore a consequence of formula~\eqref{eq:ck-character-formula-multcommutator}:
\[
\tilde\ck(x_{ij}(a))
=\multcommutator{e_{ik}((1,1))\detopone}{e_{kj}((a,P_+a P_+))\detopone}
=e_{ij}((a,P_+a P_+))\detopone
\]
as desired.

Applying formula~\eqref{eq:ck-character-formula-multcommutator} to the commutator~\eqref{eq:steinberg-symbol}, we get
\[
\tilde\ck\symb ab
  =\multcommutator{\GL(p_1)^{-1}\{d_{12}(a)\}}{\GL(p_1)^{-1}\{d_{13}(b)\}}.
\]
Plugging this into formula~\eqref{eq:ck-character-formula} yields formula~\eqref{eq:ck-character-formula-steinberg-symbols}.
\end{proof}

\section{Fredholm modules of loop operators}
\label{sec:loop-operators}

Let $\hilb_+\subset L^2(S^1)$ be the Hilbert subspace generated by $e^{in\theta}$, $n\ge0$, and let $\hilb_-$ be its orthogonal complement. Given a continuous function $q\colon S^1\to\C$, let $M(q)$ denote the multiplication operator on~$L^2(S^1)$: $M(q)\xi=q\xi$.

The following result of Pressley, Segal, and Wilson is the central observation connecting Fredholm modules to the Beilinson--Bloch regulator.

\begin{proposition}[Pressley--Segal--Wilson {\cite[2.3]{Segal-Wilson:1985}}, {\cite[6.3.1]{Pressley-Segal:1986}}]
\label{prop:loop-operator-fredholm-module}
Let $Y$ be a complex analytic manifold and let $\gamma\colon S^1\to Y$ be a smooth loop. Then with respect to the grading $L^2(S^1)=\hilb_+\oplus\hilb_-$, we have $\mult{f\circ\gamma}\in\fm1$ whenever $f\in\O(Y)$. Thus the $\C$-algebra homomorphism
\begin{equation}
\label{eq:fredholm-module-structure-on-meromorphic-functions}
\mu_\gamma\colon\O(Y)\longto\fm1,\quad
f\mapsto\mult{f\circ\gamma}
\end{equation}
endows $\O(Y)$ with the structure of a 2-summable Fredholm module.
\end{proposition}

\begin{proof}
We shall repeat the proof in \cite[Prop.~6.3.1]{Pressley-Segal:1986}, for it clarifies why $\gamma$ is required to be smooth. (In fact, continuous differentiability would suffice.)

The operator $F=P_+-P_-$ is the singular integral operator
\[
(F\xi)(\theta)=\pv\int_0^{2\pi}K(\theta,\phi)\xi(\phi)\,\frac{d\phi}{2\pi}
\]
where ``$\pv\int$'' denotes the principal value integral $\lim_{\epsilon\downarrow0}\int_0^{\theta-\epsilon}+\int_{\theta+\epsilon}^{2\pi}$, and the operator kernel~$K$ equals
\begin{equation}
\label{eq:kernel-F}
K(\theta,\phi)
=\sum_{k\geq0}e^{ik(\theta-\phi)}-\sum_{k<0}e^{ik(\theta-\phi)}
=1+i\cot\textstyle\frac12(\theta-\phi).
\end{equation}
The commutator $\addcommutator{\mult{f\circ\gamma}}F$ is therefore an integral operator whose (squared) Hilbert--Schmidt norm is
\[
\hsnorm{\addcommutator{\mult{f\circ\gamma}}F}^2
=\frac1{4\pi^2}\int_0^{2\pi}\!\!\!\int_0^{2\pi}
 \frac{|f\circ\gamma(\theta)-f\circ\gamma(\phi)|^2}
 {\sin^2\textstyle\frac12(\theta-\phi)}\,d\phi\,d\theta.
\]
Since $f\circ\gamma$ is smooth, the integrand is a continuous function on $S^1\times S^1$. Thus $\hsnorm{\addcommutator{\mult{f\circ\gamma}}F}$ is finite, and so $\addcommutator{\mult{f\circ\gamma}}F$ is Hilbert--Schmidt.
\end{proof}

\section{Reconstruction of the Beilinson--Bloch regulator}
\label{sec:reconstruction-bb-regulator}

By virtue of Proposition~\ref{prop:loop-operator-fredholm-module}, the Connes--Karoubi character~$\ck$ provides the map
\[
R_S\colon K_2(\O(X\setminus S))\times C^\infty(S^1,X\setminus S)\longto\Ctimes,
\quad
(u,\gamma)\mapsto\ck\comp K_2(\mu_\gamma)(u).
\]
The goal of this section---indeed, of the remainder of the paper---is to prove that $R_S$ coincides with the Beilinson--Bloch regulator (after passage to homotopy classes of loops). We establish this first for Steinberg symbols (\S\ref{sec:R_S-steinberg-symbols}), then for the whole group $K_2(\O(X\setminus S))$, by an argument exploiting the functorial properties of~$R_S$ (\S\ref{sec:functorial-properties}) and the continuity of~$R_S(u,\variable)$ (\S\ref{sec:continuity}).

\subsection{Computation of \texorpdfstring{$R_S$}{R\_S} on Steinberg symbols}
\label{sec:R_S-steinberg-symbols}

\begin{proposition}
\label{prop:bb-regulator-agreement-on-steinberg-symbols}
The map $R_S$ agrees with the Beilinson--Bloch regulator on Steinberg symbols: if $f$, $g\in\O(X\setminus S)^\times$ and $\gamma\colon S^1\to X\setminus S$ is a smooth loop, then
\[
\pairing{r_S\symb fg}{[\gamma]}=\ck\comp K_2(\mu_\gamma)(\symb fg).
\]
\end{proposition}

Let $\mu\colon C^\infty(S^1)\to\fm1$ be the $2$-summable Fredholm module of multiplication operators on~$L^2(S^1)$ \cite[5.7]{Connes-Karoubi:1988}: $\mu(q)=\mult q$. Let
\[
R=\ck\comp K_2(\mu)\colon K_2(C^\infty(S^1))\longto\Ctimes.
\]
(This is the character $\ck_{S^1}$ of~\S\ref{sec:related-works}.) Then
\[
R_S(\symb fg,\gamma)=R\symb{f\comp\gamma}{g\comp\gamma}.
\]
In light of Beilinson's formula~\eqref{eq:monodromy-formula}, Proposition~\ref{prop:bb-regulator-agreement-on-steinberg-symbols} is proved once the following proposition is established.

\begin{proposition}
\label{prop:R-integral-formula}
Let $p$, $q\in C^\infty(S^1)$ be nowhere vanishing functions. We have
\begin{equation}
\label{eq:R-steinberg-symbol}
R\symb pq
=\exp\left\{\frac1{2\pi i}\left(
  \int_{S^1}\log p\,d\log q-\log q(1)\int_{S^1}d\log p\right)\right\}.
\end{equation}
\end{proposition}

We will prove this by expressing each side in terms of the Fourier coefficients of $\log p$ and $\log q$, and seeing that the resulting expressions coincide. The result for the right-hand side is relatively easy to obtain.

\begin{lemma}
\label{lemma:bb-regulator-fourier-expansion}
Let $p$, $q\in C^\infty(S^1)$ be nowhere vanishing functions. We have
\[
\begin{split}
&
\exp\left\{\frac1{2\pi i}\left(
  \int_{S^1}\log p\,d\log q-\log q(1)\int_{S^1}d\log p\right)\right\}\\
&\qquad
=(-1)^{mn}\exp\biggl(
  n\ft\alpha(0)-m\ft\beta(0)+\sum_{k\in\Z}k\ft\alpha(-k)\ft\beta(k)\biggr).
\end{split}
\]
Here $m$, resp.~$n$, is the winding number of~$p$, resp.~$q$; $\alpha$, $\beta\in C^\infty(S^1)$ satisfy $e^\alpha=p/z^m$, $e^\beta=q/z^n$; $\ft\alpha$, resp.~$\ft\beta$, is the Fourier transform of~$\alpha$, resp.~$\beta$.
\end{lemma}

\begin{proof}
Develop the integrand as a Fourier series, then integrate.
\end{proof}

\subsubsection{Reduction of the computation of \texorpdfstring{$R\symb pq$}{R\{p,q\}}}

Expressing $p$ and $q$ in the notation of Lemma~\ref{lemma:bb-regulator-fourier-expansion}, we have, by the skew-symmetry and bi-multiplicativity of the Steinberg symbol,
\[
\symb pq
=\symb{z^me^\alpha}{z^ne^\beta}
=\symb zz^{mn}\symb z{e^\beta}^m\symb z{e^\alpha}^{-n}\symb{e^\alpha}{e^\beta},
\]
from which it follows that
\[
R\symb pq
=R\symb zz^{mn}
 R\symb z{e^\beta}^m
 R\symb z{e^\alpha}^{-n}
 R\symb{e^\alpha}{e^\beta}.
\]
The computation of $R\symb pq$ is thereby reduced to that of the following three cases:
\begin{itemize}
\item[(i)] $R\symb zz$,
\item[(ii)] $R\symb z{e^\alpha}$,
\item[(iii)] $R\symb{e^\alpha}{e^\beta}$.
\end{itemize}
We shall first need a general expression for $R\symb pq$.

\subsubsection{General formulas for \texorpdfstring{$R\symb pq$}{R\{p,q\}}}

\paragraph*{First formula.}

By formula~\eqref{eq:ck-character-formula-steinberg-symbols}, we have
\begin{align}
\label{eq:R-steinberg-symbol-formula-commutator-lift}
R\symb pq
&=\ck\symb{\mult p}{\mult q}\nonumber\\
&=\det\comp\GL(p_2)\bigl(\multcommutator{\GL(p_1)^{-1}\{d_{12}(\mult p)\}}{\GL(p_1)^{-1}\{d_{13}(\mult q)\}}\bigr).
\end{align}

\paragraph*{Second formula.}
Alternatively, we can give a formula for~$R\symb pq$ based on formula~\eqref{eq:ck-character-formula-lift}.
In terms of the standard generators of the Steinberg group, the Steinberg symbol $\symb{\mult p}{\mult q}$ has the expression
\begin{equation}
\label{eq:steinberg-symbol-formula}
\symb{\mult p}{\mult q}
=w_{12}(1)h_{12}(-\mult{pq})w_{12}(1)h_{12}(\mult p)h_{12}(\mult q),
\end{equation}
where $w_{12}(a)=x_{12}(a)x_{21}(-a^{-1})x_{12}(a)$ and $h_{12}(a)=w_{12}(a)w_{12}(-1)$ (see~\cite[4.2.16]{Rosenberg:1994}). Denoting the Toeplitz operator with symbol~$f$ by $\toeplitz f$, let
\[
H(p)=
\begin{pmatrix}
  \bigl(2-\toeplitz p\toeplitz{1/p}\bigr)\toeplitz p & -1+\toeplitz p\toeplitz{1/p}\\
  1-\toeplitz{1/p}\toeplitz p & \toeplitz{1/p}
\end{pmatrix},\quad
J=
\begin{pmatrix}
  \phantom{-}0 & 1\\
  -1 & 0  
\end{pmatrix}.
\]
Direct computation with formula~\eqref{eq:ck-character-formula-lift} yields
\begin{equation}
\label{eq:ck-character-steinberg-generators}
\begin{split}
&\GL(p_2)\comp\tilde\ck\bigl(w_{12}(1)\bigr)=J,\\
&\GL(p_2)\comp\tilde\ck\bigl(h_{12}(\mult p)\bigr)=H(p),\\
&\GL(p_2)\comp\tilde\ck\bigl(h_{12}(-\mult{pq})\bigr)=H(-pq).
\end{split}
\end{equation}
Plugging \eqref{eq:steinberg-symbol-formula} and \eqref{eq:ck-character-steinberg-generators} into~\eqref{eq:ck-character-formula}, we get a second formula for $R\symb pq$:
\begin{equation}
\label{eq:R-steinberg-symbol-formula}
R\symb pq
=\det\comp\GL(p_2)\comp\tilde\ck\symb{\mult p}{\mult q}
=\det\bigl(JH(-pq)JH(p)H(q)\bigr).
\end{equation}

\subsubsection{Case (i): Computation of \texorpdfstring{$R\symb zz$}{R\{z,z\}}}
\label{sec:case-i}

Let $S=\toeplitz z$ be the shift operator, and let $S^*=\toeplitz{z^{-1}}$ be its adjoint. Let $P_i\colon\hilb_+\to\hilb_+$ be the projection onto the $i$th Fourier component. Using the relations
\[
S^*S=1,\quad
SS^*=1-P_0,\quad
S^2S^*\vphantom{S}^2=1-(P_0+P_1),
\]
one finds that
\[
JH(-z^2)J=
\begin{pmatrix}
  S^*\vphantom{S}^2 & 0\\
  -(P_0+P_1) & S^2
\end{pmatrix},\quad
H(z)=
\begin{pmatrix}
  S & -P_0\\
  0 & S^*
\end{pmatrix}.
\]
Plugging these matrices into formula~\eqref{eq:R-steinberg-symbol-formula}, we find that
\begin{equation}
\begin{split}
\label{eq:R-steinberg-symbol-formula-case-one}
R\symb zz
&=\det\bigl(JH(-z^2)JH(z)^2\bigr)\\
&=\det
\begin{pmatrix}
  1 & 0\\
  0 & SP_0+P_0S^*+1-P_0-P_1
\end{pmatrix}\\
&=\det(SP_0+P_0S^*+1-P_0-P_1)\\
&=-1,
\end{split}
\end{equation}
for $SP_0+P_0S^*+1-P_0-P_1$ is the operator that transposes the first two Fourier components.

\subsubsection{Case (ii): Computation of \texorpdfstring{$R\symb z{e^\alpha}$}{R\{z,exp(alpha)\}}}
\label{sec:case-ii}

Let $S$, $S^*$, $P_0$ be as in \S\ref{sec:case-i}. By formula~\eqref{eq:R-steinberg-symbol-formula-commutator-lift}, we have
\begin{align}
\label{eq:determinant}
R\symb z{e^\alpha}
&=\det\left[
  \begin{pmatrix}
    S & P_0 & 0\\
    0 & S^* & 0\\
    0 & 0 & 1
  \end{pmatrix},
  \begin{pmatrix}
    T(e^\alpha) & 0 & 0\\
    0 & 1 & 0\\
    0 & 0 & T(e^{-\alpha})
  \end{pmatrix}
  \right]\nonumber\\
&=\det
  \begin{pmatrix}
    (ST(e^\alpha)S^*+P_0)T(e^\alpha)^{-1} & 0 & 0\\
    0 & 1 & 0\\
    0 & 0 & 1  
  \end{pmatrix}\nonumber\\
&=\det\bigl((ST(e^\alpha)S^*+P_0)T(e^\alpha)^{-1}\bigr).
\end{align}
To facilitate the evaluation of this determinant, we split $\alpha$ into its nonnegative and negative Fourier components: $\alpha=\alpha_++\alpha_-$. Then, by the multiplicativity of the Steinberg symbol, we have
\[
R\symb z{e^\alpha}=R\symb z{e^{\alpha_+}}R\symb z{e^{\alpha_-}}.
\]
We will show that
\begin{equation}
\label{eq:R-steinberg-symbol-formula-case-two}
R\symb z{e^\alpha}=\exp\bigl(-\ft\alpha(0)\bigr)
\end{equation}
by showing separately that
\[
R\symb z{e^{\alpha_+}}=\exp\bigl(-\ft\alpha(0)\bigr),
\quad
R\symb z{e^{\alpha_-}}=1.
\]

\paragraph*{Computation of $R\symb z{e^{\alpha_+}}$.}

Since $\alpha_+$ has no negative Fourier components, neither does $e^{\alpha_+}$. Consequently $T(e^{\alpha_+})$ is a multiplication operator, and so $ST(e^{\alpha_+})=T(e^{\alpha_+})S$ and $T(e^{\alpha_+})^{-1}=T(e^{-\alpha_+})$. Since $SS^*=1-P_0$, the determinant~\eqref{eq:determinant} reduces to
\begin{equation}
\label{eq:determinant-nonnegative-case}
\begin{split}
R\symb z{e^{\alpha_+}}
&=\det\bigl((ST(e^{\alpha_+})S^*+P_0)T(e^{\alpha_+})^{-1}\bigr) \\
&=\det\bigl(1-T(e^{\alpha_+})P_0T(e^{-\alpha_+})+P_0T(e^{-\alpha_+})\bigr).
\end{split}
\end{equation}
The operator $-T(e^{\alpha_+})P_0T(e^{-\alpha_+})+P_0T(e^{-\alpha_+})$ is of trace class, for it is the rank~$1$ operator ($\xi\in\hilb_+$)
\[
\xi\mapsto\ftoperator(e^{-\alpha_+}\xi)(0)(1-e^{\alpha_+})
\quad
\text{($\ftoperator$: Fourier transform).}
\]
We shall use the \emph{Grothendieck determinant formula} \cite{Grothendieck:1956}, \cite{Simon:1977}:
\[
\det(1+A)=1+\tr(\wedge^1A)+\tr(\wedge^2A)+\tr(\wedge^3A)+\cdots
\quad
\text{($A$: trace class operator)}.
\]
Let $A=-T(e^{\alpha_+})P_0T(e^{-\alpha_+})+P_0T(e^{-\alpha_+})$. We see that
\[
Az^n
=\ftoperator(e^{-\alpha_+}z^n)(0)(1-e^{\alpha_+})
=\ftoperator(e^{-\alpha_+})(-n)(1-e^{\alpha_+}),
\]
which is~$0$ unless $n=0$, because $e^{-\alpha_+}$ has no negative Fourier components. Since $A$ has rank~$1$, the higher exterior powers $\wedge^n A$ vanish. Thus Grothendieck's formula drastically collapses, reducing the determinant~\eqref{eq:determinant-nonnegative-case} to the expression
\[
R\symb z{e^{\alpha_+}}
=\det(1+A)
=1+\tr A
=1+\ftoperator(e^{-\alpha_+})(0)\bigl(1-\ftoperator(e^{\alpha_+})(0)\bigr).
\]
Evidently $\ftoperator(e^{\pm \alpha_+})(0)=\exp\bigl(\pm\ft{\alpha_+}(0)\bigr)$ (by inspection of power series in~$z$), and so
\[
R\symb z{e^{\alpha_+}}=\exp\bigl(-\ft{\alpha_+}(0)\bigr)=\exp\bigl(-\ft\alpha(0)\bigr).
\]

\paragraph*{Computation of $R\symb z{e^{\alpha_-}}$.}

We proceed as above. Observing that $\exp(\overline{\alpha_-})$ has only nonnegative Fourier components, and that $T(e^{\alpha_-})=T(e^{\overline{\alpha_-}})^*$, we see, likewise, that $T(e^{\alpha_-})S^*=S^*T(e^{\alpha_-})$ and $T(e^{\alpha_-})^{-1}=T(e^{-\alpha_-})$. The determinant~\eqref{eq:determinant} in this case reduces to the expression
\begin{equation}
\label{eq:determinant-negative-case}
\begin{split}
R\symb z{e^{\alpha_-}}
&=\det\bigl((ST(e^{\alpha_-})S^*+P_0)T(e^{\alpha_-})^{-1}\bigr) \\
&=\det\bigl(1-P_0+P_0T(e^{-\alpha_-})\bigr).
\end{split}
\end{equation}
Let $A=-P_0+P_0T(e^{-\alpha_-})$. We see that $Az^n=0$ unless $n=0$, in which case it is the constant function $-1+\ftoperator(e^{-\alpha_-})(0)$. Once again, the higher exterior powers $\wedge^n A$ vanish. Applying Grothendieck's formula to~\eqref{eq:determinant-negative-case}, we get
\[
R\symb z{e^{\alpha_-}}
=\det(1+A)
=1+\tr A
=1-1+\ftoperator(e^{-\alpha_-})(0)
=\ftoperator(e^{-\alpha_-})(0).
\]
But $\alpha_-$ has no positive Fourier components, whence $\ftoperator(e^{-\alpha_-})(0)=1$ (again evident by inspection of power series in~$z$). Consequently
\[
R\symb z{e^{\alpha_-}}=1.
\]

\subsubsection{Case (iii): Computation of \texorpdfstring{$R\symb{e^\alpha}{e^\beta}$}{R\{exp(alpha),exp(beta)\}}}
\label{sec:case-iii}

A Toeplitz operator is invertible whenever its symbol is continuous, nowhere vanishing, and of winding number zero (the Kre{\u\i}n--Widom--Devinatz Theorem; see~\cite[7.27]{Douglas:1998}). Thus $\toeplitz{e^\alpha}$ and $\toeplitz{\smash{e^\beta}\vphantom{e^\alpha}}$ are invertible ($\alpha$, $\beta\in C^\infty(S^1)$), and so we have
\begin{align*}
\begin{pmatrix}
  (\mult{e^\alpha},\toeplitz{e^\alpha}) & 0 & 0\\
  0 & (\mult{e^{-\alpha}},\toeplitz{e^{-\alpha}}) & 0\\
  0 & 0 & 1
\end{pmatrix}\detopone
&\in\GL(p_1)^{-1}\{d_{12}(\mult{e^\alpha})\},\\
\begin{pmatrix}
  (\mult{e^\beta},\toeplitz{e^\beta}) & 0 & 0\\
  0 & 1 & 0\\
  0 & 0 & (\mult{e^{-\beta}},\toeplitz{e^{-\beta}})
\end{pmatrix}\detopone
&\in\GL(p_1)^{-1}\{d_{13}(\mult{e^\beta})\}.
\end{align*}
Using these preimages, formula~\eqref{eq:R-steinberg-symbol-formula-commutator-lift} yields
\begin{align*}
R\symb{e^\alpha}{e^\beta}
&=\det\left[
\begin{pmatrix}
  \toeplitz{e^\alpha} & 0 & 0\\
  0 & \toeplitz{e^{-\alpha}} & 0 \\
  0 & 0 & 1
\end{pmatrix},
\begin{pmatrix}
  \toeplitz{e^\beta} & 0 & 0\\
  0 & 1 & 0 \\
  0 & 0 & \toeplitz{e^{-\beta}}
\end{pmatrix}\right] \\
&=\det\bigl(\toeplitz{e^\alpha}\toeplitz{e^\beta}\toeplitz{e^\alpha}^{-1}\toeplitz{e^\beta}^{-1}\bigr).
\end{align*}
The value of this determinant has been computed by Helton and Howe \cite[p.~183]{Helton-Howe:1973} to be
\begin{equation}
\label{eq:R-steinberg-symbol-formula-case-three}
\exp\tr\bigl(\toeplitz{e^\alpha}\toeplitz{e^\beta}-\toeplitz{e^\beta}\toeplitz{e^\alpha}\bigr)
=\exp\bigl(\textstyle\sum_{k\in\Z}k\ft\alpha(-k)\ft\beta(k)\bigr).
\end{equation}

\begin{proof}[Proof of Proposition~\ref{prop:R-integral-formula}]

Combining the final computations of the above three cases---\eqref{eq:R-steinberg-symbol-formula-case-one}, \eqref{eq:R-steinberg-symbol-formula-case-two}, \eqref{eq:R-steinberg-symbol-formula-case-three}---we get
\begin{align*}
R\symb pq
&=R\symb zz^{mn}
 R\symb z{e^\beta}^m
 R\symb z{e^\alpha}^{-n}
 R\symb{e^\alpha}{e^\beta}\\
&=(-1)^{mn}
 \exp\bigl(-m\ft\beta(0)\bigr)
 \exp\bigl(n\ft\alpha(0)\bigr)
 \exp\bigl(\textstyle\sum_{k\in\Z}k\ft\alpha(-k)\ft\beta(k)\bigr).
\end{align*}
This matches the right-hand side of~\eqref{eq:R-steinberg-symbol} by Lemma~\ref{lemma:bb-regulator-fourier-expansion}.
\end{proof}

\begin{remark}
An alternative computation of cases~(ii) and~(iii) is given in the thesis of the second author \cite[\S3.8]{Kasatkin:2015}.
\end{remark}

\begin{remark}
A longer, more indirect proof of equality~\eqref{eq:R-steinberg-symbol} is possible by combining several works in the operator theory literature.

Let $p\colon\fm1\to\bo/\tc$, $a\mapsto P_+a P_+\bmod\tc$. Kaad has shown \cite{Kaad:2011a}, substantiating a claim of J.~Rosenberg, that the Connes--Karoubi character~$\ck$ coincides with the composition
\[
K_2(\fm1)\xrightarrow{\,K_2(p)\,}K_2(\bo/\tc)
\xrightarrow{\,\bd\,} K_1(\bo,\tc)
\xrightarrow{\,\det\,}\Ctimes,
\]
where $\bd$ is the boundary map in $K$-theory. This is essentially the Helton--Howe \emph{determinant invariant} \cite{Helton-Howe:1973}, whose $K$-theoretic significance was recognized by Brown \cite{Brown:1973}, \cite{Brown:1975}.

In terms of the determinant invariant, we get
\begin{align*}
R_S(\symb fg,\gamma)
&=\ck\comp K_2(\mu_\gamma)\symb fg\\
&=\det\comp\bd\comp K_2(p)\symb{\mult{f\comp\gamma}}{\mult{g\comp\gamma}}\\
&=\det\comp\bd\symb{T(f\comp\gamma)\bmod\tc}{T(g\comp\gamma)\bmod\tc}.
\end{align*}
The final term here has been shown by Migler \cite{Migler:2014} to equal $\jt{\toeplitz{f\comp\gamma}}{\toeplitz{g\comp\gamma}}$, where $\jointtorsion$ is the Carey--Pincus \emph{joint torsion} \cite{Carey-Pincus:1999}, while, previously, Carey and Pincus \cite[Prop.~1]{Carey-Pincus:1999} had shown that
\begin{align*}
&\jt{\toeplitz{f\comp\gamma}}{\toeplitz{g\comp\gamma}} \\
&\qquad=\exp\left\{\frac1{2\pi i}\left(
  \int_{S^1}\log(f\comp\gamma)\,d\log(g\comp\gamma)
    -\log(g\comp\gamma)(1)\int_{S^1}d\log(f\comp\gamma)\right)\right\}.
\end{align*}
Their proof of this equality is a long and delicate computation of signs, depending ultimately on Deligne's formula for tame symbols \cite[\S2.7]{Deligne:1991} to obtain the integral expression. In this way the validity of Proposition~\ref{prop:bb-regulator-agreement-on-steinberg-symbols} is also established.

The basis of all these works, including our own, is the work of Helton and Howe. As we have shown in our proof of Proposition~\ref{prop:R-integral-formula}, the elaboration of the results of~\cite{Helton-Howe:1973} that is necessary to establish equality~\eqref{eq:R-steinberg-symbol} is neither long, delicate, nor reliant on the theory of joint torsion. Nonetheless, the Carey--Pincus approach maintains some appeal, for it derives Proposition~\ref{prop:R-integral-formula} from a broader circle of ideas.
\end{remark}

\subsection{Continuity of \texorpdfstring{$R_S$}{R\_S}}
\label{sec:continuity}

Given smooth manifolds $M$ and $N$, with $M$ compact, we endow the spaces $C^r(M,N)$ ($0\le r<\infty$) with the weak topology, and the space $C^\infty(M,N)$ with the weakest topology that makes each of the inclusions $C^\infty(M,N)\injection C^r(M,N)$ continuous (see~\cite[Ch.~2]{Hirsch:1976}). When $M$ and $N$ are given base points, the subspace of base-point preserving maps $C^\infty(M,N)_*$ is endowed with the subspace topology.

\begin{lemma}
\label{lemma:continuity-regulator-loop-variable}
Let $Y$ be a complex analytic manifold. For each $u\in K_2(\O(Y))$, the map
\[
C^\infty(S^1,Y)\longto\Ctimes,
\quad
\gamma\mapsto\ck\comp K_2(\mu_\gamma)(u)
\]
is continuous. In particular, the map $R_S(u,\variable)\colon C^\infty(S^1,X\setminus S)_*\to\Ctimes$ is continuous.
\end{lemma}

We shall give a detailed proof, since the lemma will be a key tool in the proof of the main theorem (\S\ref{sec:main-theorem}).

\begin{proof}
In terms of the the standard generators of~$\steinberg(\O(Y))$, we can write $u\in K_2(\O(Y))$ as a product
\[
u=x_1(f_1)\cdots x_N(f_N),
\quad
\hbox{for some $f_\ell\in\O(Y)$ and $x_\ell=x_{i_\ell,j_\ell}$}
\]
with the property that the corresponding product of elementary matrices $e_1(f_1)\cdots e_N(f_N)$ is~$1$. By formula~\eqref{eq:ck-character-formula-lift} we then get
\begin{equation}
\label{eq:R_S-formula-explicit}
\begin{split}
&\ck\comp K_2(\mu_\gamma)(u) \\
&\qquad=e_1\bigl((\mult{f_1\comp\gamma},\toeplitz{f_1\comp\gamma})\bigr)\cdots e_N\bigl((\mult{f_N\comp\gamma},\toeplitz{f_N\comp\gamma})\bigr)\detopone.
\end{split}
\end{equation}
Thus the map $\gamma\mapsto\ck\comp K_2(\mu_\gamma)(u)$ can be decomposed into a composition of four maps
\begin{equation}
\label{eq:S-regulator-composition}
\begin{split}
C^\infty(S^1,Y)
  &\buildrel{\rm(i)}\over\longto\textstyle\prod_1^N C^\infty(S^1) \\
  &\buildrel{\rm(ii)}\over\longto\textstyle\prod_1^N\fmext \\
  &\buildrel{\rm(iii)}\over\longto\textstyle\prod_1^N\GL(\fmext) \\
  &\buildrel{\rm(iv)}\over\longto\GL(\fmext)/\detopone
\end{split}
\end{equation}
with (final) image in $\Ctimes\isom\detop/\detopone\subset\GL(\fmext)/\detopone$ (topological embedding of~$\Ctimes$). These maps are:
\begin{itemize}
\item[(i)]
The map $\gamma\mapsto(f_1\comp\gamma,\dots,f_N\comp\gamma)$.
\item[(ii)]
The map $(\xi_1,\dots,\xi_N)\mapsto\bigl((\mult{\xi_1},\toeplitz{\xi_1}),\dots,(\mult{\xi_N},\toeplitz{\xi_N})\bigr)$.
\item[(iii)]
The map $\bigl((a_1,q_1),\dots,(a_N,q_N)\bigr)\mapsto\bigl(e_1((a_1,q_1)),\dots,e_N((a_N,q_N))\bigr)$.
\item[(iv)]
The multiplication map on $\GL(\fmext)$ followed by the projection $\GL(\fmext)\to\GL(\fmext)/\detopone$.
\end{itemize}
It is evident that the maps~(i), (iii), (iv) are continuous.

As for the map~(ii), its continuity depends on the continuity of the linear map $C^\infty(S^1)\to\fmext$, $\xi\mapsto(\mult\xi,\toeplitz\xi)$, which can be factored as $C^\infty(S^1)\injection C^1(S^1)\to\fmext$. The weak topology of $C^1(S^1)$ is induced by the Banach space norm
\[
\norm{\xi}_{C^1(S^1)}=\norm{\xi}_\infty+\norm{\dot\xi}_\infty.
\]
Since the inclusion $C^\infty(S^1)\injection C^1(S^1)$ is continuous, it suffices to show that the Banach space map $\Phi\colon C^1(S^1)\to\fmext$, $\xi\mapsto(\mult\xi,\toeplitz\xi)$, is bounded. We have
\begin{align*}
\norm{\Phi(\xi)}_{\fmext}
&=\norm{\mult\xi}_{\fm1}+\norm{P_+\mult\xi P_+-\toeplitz\xi}_{\tc}\\
&=\norm{\mult\xi}_{\fm1}\\
&=\norm{\xi}_\infty+\hsnorm{\addcommutator F{\mult\xi}}.
\end{align*}
It will therefore suffice to show that $\hsnorm{\addcommutator F{\mult\xi}}\le C\norm{\dot\xi}_\infty$, for some constant~$C>0$ that is independent of~$\xi$. Recall that the kernel of the operator $\addcommutator F{\mult\xi}$ is
\[
K_\xi(\theta,\phi)=(\theta-\phi)K(\theta,\phi)\cdot{{\xi(\theta)-\xi(\phi)}\over{\theta-\phi}},
\]
where $K(\theta,\phi)=-1-i\cot{1\over2}(\theta-\phi)$; see formula~\eqref{eq:kernel-F}. By the Mean Value Theorem, we get the bound
\[
\hsnorm{\addcommutator F{\mult\xi}}
=\norm{K_\xi}_{L^2}
\le\sup_{\theta,\phi}|(\theta-\phi)K(\theta,\phi)|\cdot\norm{\dot\xi}_\infty.
\]
Given the explicit formula for~$K$, it is clear that the supremum is finite. The continuity of the map $\prod_1^N C^\infty(S^1)\to\prod_1^N\fmext$---and consequently, that of the composition~\eqref{eq:S-regulator-composition}---is thereby established.
\end{proof}

\subsection{Functorial properties of \texorpdfstring{$R_S$}{R\_S}}
\label{sec:functorial-properties}

It will be natural to use henceforth the concrete realization
\[
\hilb=L^2(S^1,|d\theta|^{\frac12})=\hilb_+\oplus\hilb_-
\quad
\text{(square-integrable $\textstyle\frac12$-densities)}
\]
where $\hilb_+$ is the Hilbert subspace generated by $e^{in\theta}\,|d\theta|^{\frac12}$, $n\geq0$, and $\hilb_-$ is the orthogonal complement of~$\hilb_+$. The group~$\Diff^+(S^1)$ of orientation preserving diffeomorphism of~$S^1$ acts on loops $\gamma\colon S^1\to X$ by reparameterization, $(\phi^*\gamma)(\theta)=\gamma(\phi(\theta))$, and on~$\hilb$ by the formula
\[
U_\phi(\xi(\theta)\,|d\theta|^{\frac12})
=\xi(\phi^{-1}(\theta))|(\phi^{-1})'(\theta)|^{\frac12}\,|d\theta|^{\frac12}.
\]
An argument similar to that in the proof of Proposition~\ref{prop:loop-operator-fredholm-module} shows that $U_\phi$ is a unitary operator in~$\GL_1(\fm1)$; see~\cite[Prop.~6.8.2]{Pressley-Segal:1986}. (For this reason the Hilbert space of $\frac12$-densities is more natural, in the present work, than the Hilbert space of functions.)

\begin{lemma}[Reparameterization invariance]
\label{lem:reparameterization-invariance}
Let $Y$ be a complex analytic manifold. For $\phi\in\Diff^+(S^1)$ and $\gamma\colon S^1\to Y$ a smooth loop, the following diagram commutes:
\[
\begin{CD}
K_2(\O(Y)) @>{K_2(\mu_{\phi^*\gamma})}>> K_2(\fm1) \\
@VV{K_2(\mu_\gamma)}V @VV{\ck}V \\
K_2(\fm1) @>{\ck}>> \Ctimes
\end{CD}
\]
\end{lemma}

\begin{proof}
Let $f\in\O(Y)$ and let $U=U_\phi$. One finds by direct computation that
\[
\mult{f\circ\phi^*\gamma}
=U^*\mult{f\circ\gamma}U
=\ad_U(\mult{f\circ\gamma}).
\]
By functoriality we get $K_2(\mu_{\phi^*\gamma})=K_2(\ad_U)\comp K_2(\mu_\gamma)$, where we regard $\ad_U$ as an (inner) automorphism of~$\fm1$. It therefore suffices to show that $\ck\comp K_2(\ad_U)=\ck$. We will do this by appealing to the determination of~$\ck$ by the central extension~\eqref{eq:ck-central-extension}.

The homomorphism $\GL(\ad_U)\colon E(\fm1)\to E(\fm1)$ induces a map of the central extension of~$\ck$ \eqref{eq:ck-central-extension} to the central extension of $\ck\comp K_2(\ad_U)$, which is the fiber product $\GL(\fmext)/\detopone\times_{E(\fm1)}E(\fm1)$:
\[
\begin{CD}
1 @>>> \Ctimes @>>> \GL(\fmext)/\detopone\times_{E(\fm1)}E(\fm1) @>{\pi_2}>> E(\fm1) @>>> 1 \\
@. @| @VV{\pi_1}V @VV{\GL(\ad_U)}V @. \\
1 @>>> \Ctimes @>>> \GL(\fmext)/\detopone @>{\GL(p_1)}>> E(\fm1) @>>> 1
\end{CD}
\]
We need to show that the two rows are equivalent, i.e., that $\pi_1$ is an isomorphism. But this is evident because
\begin{align*}
&\GL(\fmext)/\detopone\times_{E(\fm1)}E(\fm1) \\
&\qquad=\bigl\{\,\bigl((a_{ij},q_{ij})\detopone,(b_{ij})\bigr)\in\GL(\fmext)/\detopone\times E(\fm1)\colon
  (a_{ij})=(\ad_Ub_{ij})\,\bigr\} \\
&\qquad=\bigl\{\,\bigl(x,\GL(\ad_{U^*})\comp\GL(p_1)\comp\pi_1(x)\bigr)\colon x\in\GL(\fmext)/\detopone\,\bigr\}
\end{align*}
and $\pi_1$ is the projection onto the first factor.
\end{proof}

\begin{lemma}[$S$-compatibility]
\label{lem:S-compatibility}
For $S\subset S'$, let $i_{S,S'}\colon X\setminus S'\injection X\setminus S$ be the inclusion map, and let $\res S{S'}\colon\O(X\setminus S)\to\O(X\setminus S')$ be the restriction map. Then
\[
R_S(u,i_{S,S'}\comp\gamma')=R_{S'}\bigl(K_2(\res S{S'})(u),\gamma'),
\]
whenever $u\in K_2(\O(X\setminus S))$ and $\gamma'\colon S^1\to X\setminus S'$ is a smooth loop.
\end{lemma}

\begin{proof}
As $R_S(u,\gamma)=\ck\comp K_2(\mu_\gamma)(u)$, the lemma follows from the identity $\mu_{i_{S,S'}\comp\gamma'}=\mu_{\gamma'}\comp\res S{S'}$ and functoriality of~$K_2$.
\end{proof}

\subsection{Main theorem}
\label{sec:main-theorem}

\begin{main-theorem}
The map $R_S$ coincides with the Beilinson--Bloch regulator~$r_S$: if $u\in K_2(\O(X\setminus S))$ and $\gamma\colon S^1\to X\setminus S$ is any smooth representative of $[\gamma]\in\pi_1(X\setminus S,x_0)$, then
\begin{equation}
\label{eq:connes-karoubi-equals-beilinson-bloch}
\pairing{r_S(u)}{[\gamma]}=\ck\comp K_2(\mu_\gamma)(u).
\end{equation}
In particular, the map $R_S(\variable,\gamma)$ is determined by the homotopy class of~$\gamma$, and $R_S$ is a homomorphism in both variables (upon passing to homotopy classes of loops).
\end{main-theorem}

\begin{proof}
Let $u\in K_2(\O(X\setminus S))$. Let $i_S\colon\O(X\setminus S)\to\C(X)$ be the inclusion map. By Matsumoto's Theorem \cite[4.3.15]{Rosenberg:1994}, there are functions $f_j,g_j\in\C(X)^\times$ (finite in number) such that
\[
K_2(i_S)(u)=\textstyle\prod_j\symb{f_j}{g_j}\in K_2(\C(X)).
\]
Let $S'\subset X$ to be a finite set satisfying the following two conditions:
\begin{enumerate}
\item $S'$ contains $S$ and all zeros of the $f_j$, $g_j$.
\item $K_2(\res S{S'})(u)=\prod_j\symb{f_j}{g_j}\in K_2(\O(X\setminus S'))$.
\end{enumerate}

To see that there is such a set, begin by letting $S''$ be the union of~$S$ and all the zeros of the $f_j$, $g_j$. Then $\prod_j\symb{f_j}{g_j}$ is a well-defined element of~$K_2(\O(X\setminus S''))$, so that we have, by functoriality of~$K_2$,
\[
K_2(i_{S''})\comp K_2(\res S{S''})(u)
=\textstyle\prod_j\symb{f_j}{g_j}
=K_2(i_{S''})(\textstyle\prod_j\symb{f_j}{g_j});
\]
in other words, the elements $K_2(\res S{S''})(u)$ and $\prod_j\symb{f_j}{g_j}$ of~$K_2(\O(X\setminus S''))$ have the same image in the limit $K_2(\C(X))=\injlim_S K_2(\O(X\setminus S))$. Therefore there is some (possibly larger) finite set $S'\supset S''$ in~$X$ such that
\begin{equation}
\label{eq:same-image-in-S-prime}
K_2(\res{S''}{S'})\bigl(K_2(\res S{S''})(u)\bigr)
=K_2(\res{S''}{S'})\bigl(\textstyle\prod_j\symb{f_j}{g_j}\bigr).
\end{equation}
The right-hand side of~\eqref{eq:same-image-in-S-prime} is clearly the element $\textstyle\prod_j\symb{f_j}{g_j}\in K_2(\O(X\setminus S'))$, while the left-hand side of~\eqref{eq:same-image-in-S-prime} reduces to $K_2(\res S{S'})(u)$. Hence this choice of~$S'$ satisfies the above two conditions.

Now let $\gamma\colon(S^1,1)\to(X\setminus S,x_0)$ be a smooth loop. Since reparameterizing~$\gamma$ does not change $R_S(u,\gamma)$ (Lemma~\ref{lem:reparameterization-invariance}), we may assume that~$\gamma$ is parameterized so that $x_0\notin S'$. Suppose $\gamma'\colon(S^1,1)\to(X\setminus S',x_0)$ is a smooth loop that is (smoothly) homotopic to~$\gamma$ \emph{within} $X\setminus S$. Then
\[
\begin{split}
R_S(u,\gamma')
&=R_{S'}(K_2(\res S{S'})(u),\gamma') 
&& \text{(Lemma~\ref{lem:S-compatibility})} \\
&=R_{S'}(\textstyle\prod_j\symb {f_j}{g_j},\gamma')
&& \text{(Condition~2)} \\
&=\textstyle\prod_jR_{S'}(\symb {f_j}{g_j},\gamma') \\
&=\textstyle\prod_j\pairing{r_{S'}\symb {f_j}{g_j}}{[\gamma']}
&& \text{(Proposition~\ref{prop:bb-regulator-agreement-on-steinberg-symbols})} \\
&=\pairing{r_{S'}(\textstyle\prod_j\symb {f_j}{g_j})}{[\gamma']} \\
&=\pairing{r_{S'}(K_2(\res S{S'})(u))}{[\gamma']}
&& \text{(Condition~2)} \\
&=\pairing{r_S(u)}{[\gamma]}.
\end{split}
\]
The final equality is a consequence of the known $S$-compatibility and homotopy invariance of the Beilinson--Bloch regulator.

Any neighborhood of~$\gamma$ (in the weak topology of~$C^\infty(S^1,X\setminus S)_*$) contains such a loop~$\gamma'$, for we can take $\gamma'$ to be an arbitrarily small (smooth) deformation of $\gamma$ that avoids the finite set of points of~$S'\setminus S$ that might lie on~$\gamma$. But recall that $R_S(u,\gamma)$ is continuous in~$\gamma$ (Lemma~\ref{lemma:continuity-regulator-loop-variable}). Thus $R_S(u,\gamma)$ is arbitrarily close to $\pairing{r_S(u)}{[\gamma]}=R_S(u,\gamma')$, whence $R_S(u,\gamma)=\pairing{r_S(u)}{[\gamma]}$.
\end{proof}

For the applications of the regulator to special values of $L$-functions---and more recently, to the Volume Conjecture and the quantizability criterion for curves \cite{Gukov-Sulkowski:2011}, \cite{Fuji-Gukov-Sulkowski:2012}---one passes to real coefficients:
\[
r_\eta^\R\colon K_2(\C(X))
\xrightarrow{\ r_\eta\ }\injlim\nolimits_SH^1(X\setminus S,\Ctimes)
\xrightarrow{\ \log|\variable|\ }\injlim\nolimits_SH^1(X\setminus S,\R).
\]
We record the formula for~$r_\eta^\R$ that one gets from~\eqref{eq:connes-karoubi-equals-beilinson-bloch}.

\begin{corollary}[Beilinson \cite{Beilinson:1980}]
Let $f$ and~$g$ be meromorphic functions on~$X$, and let $\gamma\colon S^1\to X$ be a smooth loop that avoids the zeros and poles of~$f$ and~$g$. We have
\begin{equation}
\label{eq:log-regulator-pairing}
\begin{split}
\pairing{r_\eta^\R\symb fg}{[\gamma]}
&=\frac1{2\pi}\int_\gamma\log f\,d(\arg g)-\log g\,d(\arg f) \\
&=\log|\ck\comp K_2(\mu_\gamma)\symb fg|.
\end{split}
\end{equation}
\end{corollary}

\section{Concluding remarks}
\label{sec:conclusion}

\subsection{An intrinsic approach to the reconstruction theorem}

Our proof of the reconstruction of the Beilinson--Bloch regulator is not ideal: we have deduced that $R_S(u,\variable)$ factors through a homomorphism $\pi_1(X\setminus S)\to\Ctimes$ by virtue of its identification with the Beilinson--Bloch regulator, whereas only an intrinsic proof of this fact could be deemed truly satisfactory. Thus while the expression of the Beilinson--Bloch regulator in terms of the Connes--Karoubi character is itself canonical, it cannot (yet) be construed as a completely independent approach to the regulator.

A more satisfactory approach to the reconstruction theorem would entail giving intrinsic proofs of the following facts:
\begin{description}
\item[\normalfont\scshape Homotopy invariance:]
\slshape Whenever $\gamma_0$, $\gamma_1$ are homotopic we have $R_S(u,\gamma_0)=R_S(u,\gamma_1)$.
\item[\normalfont\scshape Multiplicativity:]
\slshape The induced map $R_S(u,\variable)\colon\pi_1(X\setminus S)\to\Ctimes$ is a homomorphism.
\end{description}

Concerning homotopy invariance, one might try to prove it by appealing to Karoubi's geometric characterization of elements of~$K_2(A)$ as virtual flat $A$-bundles over the sphere~$S^2$ \cite[3.11]{Karoubi:1987}, in conjunction with the geometric interpretation of the Connes--Karoubi character \cite[4.10]{Connes-Karoubi:1988}. Alternatively, given a smooth one-parameter family $\gamma_t$ of (based) loops in~$X\setminus S$, one might try to use the explicit formula~\eqref{eq:R_S-formula-explicit} to show that $\frac{d}{dt}\bigl(\ck\comp K_2(\mu_{\gamma_t})(u)\bigr)=0$. We have been unable to carry out either proposal.

Granting homotopy invariance and multiplicativity, it would then follow from Lemma~\ref{lem:S-compatibility} that, whenever $S\subset S'$, the diagram
\[
\begin{CD}
K_2(\O(X\setminus S)) @>{K_2(\res S{S'})}>> K_2(\O(X\setminus S')) \\
@VV{R_S}V @VV{R_{S'}}V \\
\Hom(\pi_1(X\setminus S,x_0'),\Ctimes) @>{(i_{S,S'})_*}>> \Hom(\pi_1(X\setminus S',x_0'),\Ctimes)
\end{CD}
\]
commutes (notation as in Lemma~\ref{lem:S-compatibility}). In this way a homomorphism
\[
R_\eta\colon K_2(\C(X))\longto H^1(\Spec\C(X),\Ctimes)
\]
would be induced.

That $R_\eta$ coincides with the Beilinson--Bloch regulator~\eqref{eq:regulator-generic-point} would follow again from Proposition~\ref{prop:bb-regulator-agreement-on-steinberg-symbols}. It is unclear to us whether it might be possible to obviate even that, by appealing to the work of Feliu \cite{Feliu:2011} and Gillet \cite{Gillet:1981} on axiomatic characterizations of Chern character maps.

\subsection{The Heisenberg group and the regulator}
\label{sec:heisenberg-group-construction}

There is a beautiful geometric construction of the Beilinson--Bloch regulator due to Bloch, Ramakrishnan, and Deligne \cite{Bloch:1981}, \cite{Ramakrishnan:1981}. We give a rapid sketch of it.

For a ring~$A$, let $H(A)$ be the Heisenberg group in three variables:
\[
H(A)=\left\{\,\left(
\begin{smallmatrix}
1 & x & z \\
0 & 1 & y \\
0 & 0 & 1
\end{smallmatrix}\right)\colon x,\ y,\ z\in A\,\right\}.
\]
The complex $3$-manifold $H(\Z)\bs H(\C)$ has a projection to~$\Ctimes\times\Ctimes$,
\[
H(\Z)\left(
\begin{smallmatrix}
1 & x & z \\
0 & 1 & y \\
0 & 0 & 1
\end{smallmatrix}\right)
\mapsto (e^{2\pi ix},e^{2\pi iy}),
\]
making $H(\Z)\setminus H(\C)$ into a principal $\Ctimes$-bundle over~$\Ctimes\times\Ctimes$. This bundle admits a canonical holomorphic connection~$\cnx$ with curvature $(du/u)\wedge(dv/v)$ (in terms of coordinates on $\Ctimes\times\Ctimes$). If $f$ and~$g$ are meromorphic functions on~$X$, and if $S\subset X$ is a finite set containing the zeros and poles of~$f$ and~$g$, then the map $(f,g)\colon X\setminus S\to\Ctimes\times\Ctimes$ pulls back the principal $\Ctimes$-bundle $(H(\Z)\setminus H(\C),\cnx)$ to a principal $\Ctimes$-bundle on~$X\setminus S$ whose associated line bundle $r\symb fg$ is flat, and may therefore be regarded as an element of~$H^1(X\setminus S,\Ctimes)$. The bundle $r\symb fg$ satisfies all the (Steinberg) symbol properties with respect to~$f$ and $g$---the only nontrivial one being the Steinberg relation that $r\symb f{1-f}$ is the trivial line bundle (whenever $1-f$ is also nowhere vanishing on~$X\setminus S$). Functoriality and the naturality of pull-backs reduces this assertion to the triviality of $r\symb f{1-f}$ in the case $X=P^1(\C)$, $S=\{0,1,\infty\}$, $f=\text{identity}$; this is established by an essential use of the dilogarithm function $\sum_{n=1}^\infty z^n/n^2$. Thus $r$ defines a homomorphism $K_2(\C(X))\to\injlim_S H^1(X\setminus S,\Ctimes)$. To see that this again recovers the Beilinson--Bloch regulator, the monodromy of~$r\symb fg$ around a loop~$\gamma$ in~$X\setminus S$ is computed and seen to recover Beilinson's formula~\eqref{eq:monodromy-formula}.

It would be desirable to understand, concretely, the relationship between this geometric description of the regulator and our operator-theoretic one.

\subsection{Entropy, regulators, determinants (Deninger's problem)}
\label{sec:entropy-regulators-determinants}

Lastly, let us sketch our original motivation for undertaking a construction of the Beilinson--Bloch regulator in the Fredholm-module context, that is, in the setting of noncommutative geometry. Though the issues we raise will be somewhat vaguely and imprecisely posed, we believe they warrant a closer examination, nonetheless.

Deninger has established curious formulas equating the entropy of~$\Z^d$-actions, regulator pairings (of the kind that we have consider in this article), Mahler measure, and the Fuglede--Kadison determinant for the von~Neumann algebra of~$\Z^d$ \cite[\S2.4]{Deninger:2012}. An example of this is the following assertion. Let $\torus^d$ be the (real) $d$-torus, and let $P\in\Z[\Z^d]$ with $P\in\ell^1(\Z^d)^\times$. Deninger shows~\cite[Cor.~11]{Deninger:2012} that
\begin{equation}
\label{eq:deningers-strange-formula}
\log\det\nolimits_{\vna\Z^d}(P)
=h(P)
=\pairing{r\{P,e_1,\dots,e_d\}}{(2\pi i)^{-d}[\torus^d]}.
\end{equation}
Here $\det\nolimits_{\vna\Z^d}$ is the Fuglede--Kadison determinant of the group von~Neumann algebra~$\vna\Z^d$; $h(P)$ is the topological entropy of the action of~$P$ on the Pontrjagin dual of the discrete group $\Z[\Z^d]/\Z[\Z^d]P$; $r$ is the Beilinson regulator and $r\{P,e_1,\dots,e_d\}$ is a class in the (singular) cohomology $H^d\bigl((\Ctimes)^d\setminus\{P=0\},\R(d)\bigr)$; $[\torus^d]$ is the fundamental class in homology.

In lieu of a description of the Beilinson regulator~$r$ and of the class~$r\{P,e_1,\dots,e_d\}$ (see~\cite[\S2.4]{Deninger:2012} and~\cite[\S1]{Deninger:1997}), one may still gain insight into the substance of the equalities~\eqref{eq:deningers-strange-formula} by examining the case~$d=1$:
\begin{equation}
\label{eq:deningers-strange-formula-special-case}
\begin{split}
\log\det\nolimits_{\vna\Z}(P)
&=\frac1{2\pi i}\int_\torus\log|P(z)|\,\frac{dz}z \\
&=\pairing{r\symb P{e_1}}{(2\pi i)^{-1}[\torus]}
&&\text{\cite[p.~276]{Deninger:1997}} \\
&=\log|\ck\comp K_2(\mu)\symb Pz|
&&\text{(formula~\eqref{eq:log-regulator-pairing}).}
\end{split}
\end{equation}

From the work of Li and Thom \cite{Li-Thom:2012} it is known that the first equality of~\eqref{eq:deningers-strange-formula} is valid in much greater generality: it remains true when $\Z^d$ is replaced by an arbitrary countable discrete amenable group~$\Gamma$ and $P\in\Z\Gamma$ is any unit in~$\vna\Gamma$. But in this generality little is understood about what should then correspond to the right-hand term of~\eqref{eq:deningers-strange-formula}---neither the regulator nor the pairing. Deninger poses the problem of making sense of the right-hand side of~\eqref{eq:deningers-strange-formula} for certain \emph{noncommutative}~$\Gamma$, for example those that are polycyclic.

An interesting example of such a group is the integral Heisenberg group~$H(\Z)$, which we encountered above. Since $H(\Z)$ is isomorphic to the semi-direct product $\Z\ltimes\Z^2$, the group von~Neumann algebra $\vna H(\Z)$ has a direct integral decomposition
\begin{equation}
\label{eq:vna-heisenberg-group}
\vna H(\Z)
\isom\textstyle\int^\oplus_\torus L^\infty(\torus)\rtimes_\theta\Z\,d\theta.
\end{equation}
Geometrically, one views this decomposition as an interpretation of $\vna H(\Z)$ as the $L^\infty$-algebra of the space of noncommutative tori fibered over~$\torus$; indeed, this is consistent with the interpretation of noncommutative tori as degenerations of elliptic curves $E_q=\Ctimes\!/q^\Z$ as $|q|\to1$.

Consideration of the decomposition~\eqref{eq:vna-heisenberg-group} in conjunction with formula~\eqref{eq:deningers-strange-formula-special-case} suggests that in order to give meaning for~$\vna H(\Z)$ of the regulator pairing in~\eqref{eq:deningers-strange-formula}, one should seek an analogue of the Beilinson--Bloch regulator---at least at the level of a regulator pairing for special $K$-theory elements---for noncommutative tori as \emph{smooth}, and not simply as measurable, noncommutative spaces.

To this end, an analogue for noncommutative tori of the structure sheaf~$\O$ may be required. Soibelman and Vologodsky \cite{Soibelman-Vologodsky:2003} have defined a reasonable candidate for the category of coherent sheaves on a noncommutative elliptic curve $\Ctimes/q^\Z$, $|q|=1$: it is the category of modules over the crossed-product algebra $\O(\Ctimes)\crossprod q^\Z$ that are finitely presentable over~$\O(\Ctimes)$. Together with the Soibelman--Vologodsky theory, our Fredholm-module framework for the Beilinson--Bloch regulator might therefore provide an appropriate setting for the study of Deninger's problem for $\Gamma=H(\Z)$.

\bibliographystyle{plain}
\small
\bibliography{ehvk}

\end{document}